\documentclass[12pt]{amsart}
\usepackage{amssymb,verbatim,enumerate,ifthen}
\usepackage[mathscr]{eucal}
\usepackage[utf8]{inputenc}
\usepackage[T1]{fontenc}
\usepackage{cite}
\makeatletter
\@namedef{subjclassname@2010}{%
  \textup{2010} Mathematics Subject Classification}
\makeatother

\newtheorem{thm}{Theorem}

\newtheorem{prop}{Proposition}
\newtheorem{lem}{Lemma}
\newtheorem{rem}{Remark}

\newtheorem*{xrem}{Remark}
\numberwithin{equation}{section}

\newcommand\R{\mathbb{R}}
\newcommand\N{\mathbb{N}}
\newcommand\Q{\mathbb{Q}}
\newcommand\Z{\mathbb{Z}}
\newcommand{\norm}[1]{\left\| #1 \right\| }
\newcommand{\abs}[1]{\left| #1 \right| }
\newcommand\Lo{\mathscr{L}}
\newcommand\Tr{\mathscr{T}}
\newcommand\Up{\mathscr{U}}
\newcommand\Bo{\mathscr{B}}

\numberwithin{equation}{section}
\def\eq#1{{\rm(\ref{#1})}}
\def\Eq#1#2{\ifthenelse{\equal{#1}{*}}
  {\begin{equation*}\begin{aligned}[]#2\end{aligned}\end{equation*}}
  {\begin{equation}\begin{aligned}[]\label{#1}#2\end{aligned}\end{equation}}}
  
\author[P. Pasteczka]{Pawe\l{} Pasteczka}
\address{Institute of Mathematics \\ Pedagogical University of Cracow \\ Pod\-cho\-r\k{a}\-\.zych str. 2, 30-084 Krak\'ow, Poland}
\email{pawel.pasteczka@up.krakow.pl}

\title[Discontinuous mean-type mappings]{Invariant property for discontinuous mean-type mappings}

\begin{document}
\maketitle
\begin{abstract}
It is known that if $M,\,N$ are continuous two-variable means such that
$\abs{M(x,y)-N(x,y)} < \abs{x-y}$ for every $x,\ y$ with $x\ne y$,
then there exists a unique invariant mean (which is continuous too).

We are looking for invariant means for pairs satisfying the inequality above, but continuity of means is not assumed. 

In this setting the invariant mean is no longer uniquely defined, but we prove that there exist the smallest and the biggest one. Furthermore it is shown that there exists at most one continuous invariant mean related to each pair.
\end{abstract}

\section{Introduction}
The idea of invariant means was first introduced by Gauss \cite{Gau18} who considered so-called arithmetic-geometric means.
It was obtained as a limit in the iteration process 
\Eq{*}{
x_{n+1}=\frac{x_n+y_n}{2},\qquad y_{n+1}=\sqrt{x_ny_n} \qquad (n \in \N_+\cup\{0\}),
}
where $x_0,\,y_0$ are two positive arguments. Then it is known that both $(x_n)$ and $(y_n)$ are convergent to a common limit which called the \emph{arithmetic-geometric mean} (of the initial arguments $x_0:=x$ and $y_0:=y$). 

In a more general setting a \emph{mean} is an arbitrary function $M \colon I^2 \to I$ (from now on $I$ stands for an arbitrary interval) such that 
\Eq{*}{
\min(x,y) \le M(x,y) \le \max(x,y)\quad\text{ for all }\quad x,\,y \in I.
}
If inequalities above remains strict unless $x=y$, then the mean $M$ itself is called \emph{strict}. 

For two means $M,\,N$ on $I$ we define a selfmapping $(M,\,N) \colon I^2 \to I^2$ by 
$(M,\,N)(x,y):=(M(x,y),N(x,y))$. We call a mean $K$ on $I$ to be an \emph{$(M,\,N)$-invariant mean} if $K=K \circ (M,N)$; more precisely
\Eq{*}{
K(x,y)=K\big(M(x,y),N(x,y)\big)\quad\text{ for all }\quad x,\,y \in I.
}

In this setting the arithmetic-geometric mean is an invariant mean for arithmetic and geometric mean. In fact it was proved \cite[Theorem~8.2]{BorBor87} that if $M$ and $N$ are continuous and strict then such $K$ always exists and is uniquely determined. Later Matkowski \cite{Mat99b} proved that the strictness assumption can be relax to
\Eq{weakIn}{
\abs{M(x,y)-N(x,y)}< \abs{x-y} \text{ for all }x,\,y \in I,\:x \ne y.
}

Finally, similarly like in the case of arithmetic-geometric mean we know (see e.g. \cite{BorBor87}) that the $(M,\,N)$-invariant mean is obtained as a common limit of iterates of the mean-type mapping $(M,\,N)$ given by
\Eq{E:xnyn}{
x_0&=x, &\qquad y_0&=y;\\
x_{n+1}&=M(x_n,y_n), &\qquad y_{n+1}&=N(x_n,y_n) \qquad\text{ for all }n \ge0;
}
where $x$ and $y$ are its arguments. In fact these sequences of iterates are used so often that whenever the quadruple $(M,N,x,y)$ is defined, sequences $(x_n)$ and $(y_n)$ are also given.

Invariant means were extensively studied during recent years, see for example papers by Baj\'ak--P\'ales \cite{BajPal09b,BajPal09a,BajPal10,BajPal13}, by Dar\'oczy--P\'ales \cite{Dar05a,DarPal02c,DarPal03a}, by G{\l}azowska \cite{Gla11b,Gla11a}, by Matkowski \cite{Mat99b,Mat02b,Mat05,Mat06b,Mat13}, by Matkowski--P\'ales \cite{MatPal15}, by the author \cite{Pas15c,Pas16a,Pas1801}, and in the seminal book Borwein-Borwein \cite{BorBor87}.

We will consider $(M,\,N)$-invariant means where $M,\,N$ satisfies inequality \eq{weakIn} but continuity is replaced by symmetry (i.e. $M(x,y)=M(y,x)$ for all $x,\,y \in I$). 

Let us just mention that we do not require the means to be discontinuous. On the other hand if both of them are continuous then our consideration reduces to the one which was already done many times (see references above).

\section{Invariant means with no continuity assumption}

In this section we are going to present some examples of constructions which provide $(M,\,N)$-invariant means, where $M$ and $N$ are not necessarily continuous. 

There are two somehow independent ways of defining such means. First idea is to extend the meaning of limit which appear in the definition of invariant mean (for example to $\liminf$ or $\limsup$).  We realize this idea in section~\ref{sec:BIM}. Second one is related with transfinite iterations (section~\ref{sec:TIM}).

Let us begin with two elementary, however useful, results
\begin{lem}
 If $M,N \colon I^2 \to I$ are symmetric means then every $(M,\,N)$-invariant mean is symmetric.
\end{lem}
Indeed, if $K$ is an arbitrary $(M,\,N)$-invariant mean then for every $x,\,y \in I$ we get
 \Eq{*}{
 K(x,y)=K(M(x,y),N(x,y))=K(M(y,x),N(y,x))=K(y,x).
 }
\begin{lem}
\label{lem:symme}
 If $M,N \colon I^2 \to I$ are symmetric means then a mean is $(M,\,N)$-invariant if and only if it is $(M\wedge N,M \vee N)$-invariant, where 
 \Eq{*}{
 (M\wedge N)(x,y):=\min(M(x,y),N(x,y)),\quad x,\,y \in I,\\
 (M\vee N)(x,y):=\max(M(x,y),N(x,y)),\quad x,\,y \in I.
 } 
\end{lem}
By the previous lemma every $(M,\,N)$-invariant (or $(M\wedge N,M \vee N)$-invariant) mean is symmetric. Furthermore for every symmetric function $K \colon I^2 \to I$ we have $K\circ(M,N)=K \circ (M\wedge N,M \vee N)$.

\subsection{\label{sec:BIM} Boundary invariant means} 
This idea is motivated by generalized limit function. Our consideration covers all standard type of limits (i.e. $\lim$, $\liminf$, $\limsup$) but also more general functionals like Banach limit\footnote{Banach limit is a linear functional $L \colon \ell^\infty \to \R$ such that $\norm{L}_{\infty}=1$; $L(a_2,a_3,\dots)=L(a)$ for every $a \in \ell^\infty$; $L(a) \ge 0$ whenever $a_n \ge 0$ for all $n$; and $L(a)=\lim_{n \to \infty} a_n$ for every convergent sequence $(a_n)$ (cf. Conway \cite{Con90}).}. A function $\phi \colon \ell^{\infty}(I) \to I$ is called 
\emph{2-limit-like} if for every $a=(a_1,a_2,\dots) \in \ell^{\infty}(I)$ 
\begin{enumerate}[(i)]
\item $\phi(a_1,a_2,a_3,\dots)=\phi(a_3,a_4,a_5,\dots)$, and
\item $\liminf_{n \to \infty} a_n \le \phi(a_1,a_2,\dots)\le \limsup_{n \to \infty} a_n$.
 \end{enumerate}
Note that whenever the sequence $a$ is convergent, then $\phi(a)= \lim_{n \to \infty} a_n$. 

Let us emphasize that 2-limit-like function are much more general objects than common (or even Banach) limits. In fact we can construct $2^{\mathfrak{c}}$ different 2-limit-like functions. Indeed, each function $w \colon [0,1] \to [0,1]$ lead to a 2-limit-like function on $\ell^{\infty}[0,1]$ given by
\Eq{*}{
\phi_w(a):= \liminf_{n \to \infty} a_{n} + w\big(\liminf_{n \to \infty} a_{2n}\big) \cdot \Big( \limsup_{n \to \infty} a_{n}- \liminf_{n \to \infty} a_{n} \Big).
}
Furthermore, by taking a family of $4$-periodic sequences $(0,x,0,1,\dots)$ for $x \in [0,1]$, it can be verified that the mapping $w \mapsto \phi_w$ is one-to-one.

At the moment we can use this definition to introduce the wide class of $(M,\,N)$-invariant means.
\begin{prop}
Let $M,\,N \colon I^2 \to I$ be two means and $\phi  \colon \ell^{\infty}(I) \to I$ be a 2-limit-like function. Then the mean $\Bo_\phi$ given by
\Eq{*}{
\Bo_\phi(x,y):=\phi(x_0,y_0,x_1,y_1,x_2,y_2,\dots)
}
is $(M,\,N)$-invariant.

Conversely, every $(M,\,N)$-invariant mean equals $\Bo_\phi$ for some 2-limit-like function $\phi$.
\end{prop}
\begin{proof}
 By the definition of mean we have, for all $n \ge 0$,
 \Eq{*}{
 \max(x_{n+1},y_{n+1}) \le \max(x_n,y_n).
 }
Thus the sequence $(\max(x_{n},y_{n}))_{n \in \N}$ is nondecreasing and 
 \Eq{*}{
 \limsup \:(x_0,y_0,x_1,y_1,x_2,y_2,\dots)
 &=\limsup_{n \to \infty} \max(x_n,y_n)\\
 &\le \max(x_0,y_0)=\max(x,y).
 }
Similarly we obtain $\liminf \:(x_0,y_0,x_1,y_1,x_2,y_2,\dots)\ge\min(x,y)$.

Now, as $\phi$ is between $\liminf$ and $\limsup$, we obtain that $\Bo_\phi$ is a mean. Moreover 
\Eq{*}{
\Bo_\phi(M(x,y),N(x,y))&=\phi\Big(M(x_0,y_0),N(x_0,y_0),\\
&\qquad\hspace{-4em} M\big(M(x_0,y_0),N(x_0,y_0)\big),
 N\big(M(x_0,y_0),N(x_0,y_0)\big),\dots\Big)\\
&=\phi(x_1,y_1,x_2,y_2,\dots)
=\phi(x_0,y_0,x_1,y_1,x_2,y_2,\dots)\\
&=\Bo_\phi(x_0,y_0)=\Bo_\phi(x,y),
}
which concludes the proof.

To prove the converse, for an be an arbitrary $(M,\,N)$-invariant mean $K$,
we define function $\phi$ on the orbit of $(x,\,y)$ by
\Eq{proper}{
\phi(x_0,y_0,x_1,y_1\dots):=K(x,\,y) &\qquad x,\,y \in I,
}
fulfilled by 
\Eq{fulfilled}{
\phi(a_1,a_2,a_3,a_4,\dots)=\liminf_{n \to \infty} a_n.
}

By the definition  of sequences $(x_n),\,(y_n)$ and elementary properties of $\liminf$ we obtain that $\phi$ satisfies (i). Moreover, in view of (i) and the easy-to-check inequality $\inf(a) \le \phi(a)\le \sup(a)$, the property (ii) is also valid.
\end{proof}

In two particular cases $\phi=\liminf$ and $\phi=\limsup$, as
\Eq{*}{
[\min(M(x,y),N(x,y)),\,\max(M(x,y),N(x,y))] \subset [x,\,y]
}
is valid for every $x,\,y \in I$ with $x <y$, we obtain two very important $(M,\,N)$-invariant means.
Define \emph{lower-} and \emph{upper-invariant means} $\Lo,\,\Up \colon I^2 \to I$ by
\Eq{*}{
\Lo(x,y)&:=\Bo_{\liminf}(x,y)=\lim_{n \to \infty} \min(x_n,y_n), \\
\Up(x,y)&:=\Bo_{\limsup}(x,y)=\lim_{n \to \infty} \max(x_n,y_n).
}

In fact $\Lo$ and $\Up$ are the smallest and the greatest $(M,\,N)$-invariant means, respectively, as every $(M,\,N)$-invariant mean is bounded from below by $\min(x_n,y_n)$ and from above by $\max(x_n,y_n)$ (for all $n \in \N$).

\subsection{\label{sec:TIM}Transfinite invariant mean}

 Transfinite invariant mean is the third (after lower- and upper-) natural invariant mean. In order to define it we assume comparability of means $M$ and $N$ --- more precisely $M(x,y) \le N(x,y)$ for all $x,\,y \in I$. Moreover we assume that the inequality \eq{weakIn} is valid.  

Let us consider two transfinite sequences\footnote{that is sequences which are enumerated by ordinal numbers; cf. Cantor \cite{Can55}.}
$(x_\alpha)$ and $(y_\alpha)$ by fulfilling convention \eq{E:xnyn} in the following way
\Eq{E:defxalyal}{
x_{\alpha}:=\lim_{\beta\nearrow\alpha} x_\beta, \qquad y_{\alpha}:=\lim_{\beta\nearrow\alpha} y_\beta \qquad\text{ for all limit ordinals }\alpha.
}
To provide the correctness of this definition we observe that $(x_\alpha)$ is nondecreasing while $(y_\alpha)$ is nonincreasing. Still, whenever $M$, $N$, $x$, and $y$ are given, these sequences are automatically provided.

Inequality $M \le N$ implies that $x_\alpha \le y_\alpha$ for every $\alpha>0$. In particular, by  the definition of $\Lo$ and $\Up$, we get 
\Eq{E:LoUpomega}{
\Lo(x,\,y)= x_{\omega}\quad\text{ and }\quad\Up(x,\,y)=y_{\omega}\,.
}
Thus 
\Eq{*}{
A_\alpha\colon I^2 \ni (x,y)\mapsto x_\alpha \in I \qquad \text{and}\qquad B_\alpha\colon I^2 \ni (x,y)\mapsto y_\alpha \in I
} 
are expressed as a function of $\Lo(x,\,y)$ and $\Up(x,\,y)$ for all $\alpha>\omega$. In particular they are all $(M,N)$-invariant. Moreover $A_\omega=\Lo$ and $B_\omega=\Up$.

The next lemma shows that iteration sequences $(A_\alpha)$ and $(B_\alpha)$ are eventually fixed. They reach that state after at most $\omega_1$ iterations ($\omega_1$ stands for the first uncountable ordinal). This imply that there is no point to consider indexes greater than $\omega_1$ as no new means are obtained.

\begin{lem}
 Let $M,\,N \colon I^2 \to I$ be two means having property \eq{weakIn} such that $M\le N$. Then $A_{\omega_1}(x,y)=B_{\omega_1}(x,y)$ for all $x,\,y \in I$.
\end{lem}

\begin{proof}
We need to prove that $x_{\omega_1}=y_{\omega_1}$.
Inequality $M \le N$ implies $x_\alpha \le y_\alpha$ for all $\alpha \ge 1$. Moreover  \eq{weakIn} yields that for every $\alpha<\omega_1$ either $y_\alpha=x_\alpha$ (equivalently $y_\alpha-x_\alpha=0$) or $y_{\alpha+1}-x_{\alpha+1} < y_{\alpha}-x_{\alpha}$.

If $x_{\alpha_0}=y_{\alpha_0}$ for some $\alpha_0<\omega_1$ then by reflexivity of mean we obtain $x_{\alpha}=y_{\alpha}$ for all $\alpha \in [\alpha_0,\omega_1]$. In particular $x_{\omega_1}=y_{\omega_1}$.

From now on we may assume that $(y_\alpha-x_\alpha)_{\alpha<\omega_1}$ is strictly decreasing. As $x_\alpha \le y_\alpha$ we know know that this sequence consists of nonnegative entries only. This lead to a contradiction as every strictly decreasing sequence of nonnegative numbers is countable.
\end{proof}

\begin{rem}
As both $M$ and $N$ are means we obtain, applying the inequality $M \le N$, that the sequence
$(x_\alpha)_{\alpha \le \omega_1}$ is nondecreasing, while $(y_\alpha)_{\alpha \le \omega_1}$ is nonincreasing.
\end{rem}

Based on the lemma above we can define, for $M \le N$, a \emph{transfinite invariant mean} $\Tr \colon I^2 \to I$ by 
\Eq{def:Trdev}{
\Tr(x,y):=A_{\omega_1}(x,y)=B_{\omega_1}(x,y).
}

By the virtue of Lemma~\ref{lem:symme}, we can skip the comparability assumption whenever both means are symmetric (like it was already done in the case of $\Lo$ and $\Up$). 

Let us now present some important property of transfinite invariant mean.

\begin{thm} \label{thm:uniqcont}
Let $I$ be an interval, $M,\,N \colon I^2 \to I$ be means with $M\le N$ satisfying \eq{weakIn}. Either $\Tr$ is a unique continuous $(M,\,N)$-invariant mean or there are no continuous $(M,\,N)$-invariant means.
\end{thm}
\begin{proof}
Let $K$ be an arbitrary continuous $(M,\,N)$-invariant mean. We show that $K=\Tr$.

Fix $x,\,y \in I$. Using the definition of $\Tr$, it suffices to prove that $K(x,y)=x_{\omega_1}$. We will proof by transfinite induction that 
\Eq{E:TrIndas}{
K(x,y)=K(x_\alpha,y_\alpha)\quad\text{ for all }\quad\alpha\ge 0.
}

Indeed, as $K$ is $(M,\,N)$-invariant, we obtain 
\Eq{*}{
K(x_{\alpha+1},y_{\alpha+1})=K\big(M(x_\alpha,y_\alpha),N(x_\alpha,y_\alpha)\big)=K(x_\alpha,y_\alpha).
}
Furthermore, as $K$ is continuous, for every limit ordinal number $\alpha$, we get
\Eq{*}{
K(x_\alpha,\,y_\alpha )
=K\Big(\lim_{\beta \nearrow \alpha} x_\beta,\,\lim_{\beta \nearrow \alpha} y_\beta \Big)
=\lim_{\beta \nearrow \alpha} K( x_\beta,\,y_\beta).
}
Now \eq{E:TrIndas} easily follows. Finally, reflexivity of $K$ binded with equality $x_{\omega_1}=y_{\omega_1}$ concludes the proof.
\end{proof}
\begin{xrem}
By Lemma~\ref{lem:symme}, we can skip comparability assumption whenever both $M$ and $N$ are symmetric.
\end{xrem}

\section{Application and conclusions}

\subsection{Example of invariant property for noncontinuous means}
Fix an interval $I$ with $|I|>1$ and functions $M,\,N \colon I^2 \to I$ defined by
\Eq{*}{
M(x,y)&:=\begin{cases} \frac12(x+y) & \text{ for } \abs{x-y} \le 1, \\ \frac12 \big(x+y-\sqrt{\abs{x-y}}\:\big) & \text{ for }\abs{x-y} > 1, \end{cases}\qquad x,\,y \in I;\\
N(x,y)&:=\begin{cases} \frac12(x+y) & \text{ for } \abs{x-y} \le 1, \\ \frac12 \big(x+y+\sqrt{\abs{x-y}}\:\big) & \text{ for }\abs{x-y} > 1, \end{cases}\qquad x,\,y \in I.
}

It is easy to check that both $M$ and $N$ are symmetric and strict means on $I$. Furthermore the arithmetic mean is $(M,\,N)$-invariant. Whence, by Theorem~\ref{thm:uniqcont}, it is a transfinite invariant mean for this pair.

Let $(x_\alpha)$ and $(y_\alpha)$ are two transfinite sequences corresponding to the iteration $(M,\,N)$. Obviously, as $N \ge M$, we have $y_\alpha \ge x_\alpha$ for all $\alpha>0$. Thus, for all $\alpha\ge 0$,
\Eq{*}{
y_{\alpha+1}-x_{\alpha+1}=\begin{cases} 
0 & \text{ if }\abs{y_\alpha-x_\alpha}\le 1,\\
\sqrt{\abs{y_\alpha-x_\alpha}} & \text{ if }\abs{y_\alpha-x_\alpha}>1.
                                \end{cases}
}
However the iteration of square root is well known, so we obtain
\Eq{Ex1}{
y_{\omega}-x_{\omega}=\begin{cases} 0 & \text{ if }\abs{x-y}\le1, \\
                                 1 & \text{ if }\abs{x-y}> 1.
                                \end{cases}
}
On the other hand we can check by simple induction that 
\Eq{Ex2}{
x_\alpha+y_\alpha=x+y\quad \text{ for all }\alpha \ge 0.
}
We now bind \eq{E:LoUpomega}, \eq{Ex1}, and \eq{Ex2} for $\alpha=\omega$ to obtain
\Eq{*}{
\Lo(x,y)&=\begin{cases} \frac{x+y}2 & \text{ if }\abs{x-y}\le1, \\
                                 \frac{x+y-1}2 & \text{ if }\abs{x-y}> 1;
                                \end{cases} \\
\Up(x,y)&=\begin{cases} \frac{x+y}2 & \text{ if }\abs{x-y}\le1, \\
                                 \frac{x+y+1}2 & \text{ if }\abs{x-y}> 1.
                                \end{cases} 
}
To express it briefly, for every $c \in [-1,1]$, define the mean $K_c \colon I^2 \to I$ by
\Eq{*}{
K_c(x,y):=\begin{cases} \frac{x+y}2 & \abs{x-y} \le 1, \\ 
\frac{x+y+c}2 & \abs{x-y} > 1.
\end{cases}
}
Having this new notation we can simply write $\Lo=K_{-1}$, $\Up=K_1$, and $\Tr=K_0$.

If we now continue inductive steps we get $A_{\omega+1}=B_{\omega+1}=K_0=\Tr$. Thus  (in this example) sequences $(A_\alpha)_{\alpha\ge \omega}$ and $(B_\alpha)_{\alpha\ge \omega}$ contain the lower-, upper-, and transfinite- invariant means only. 

On the other hand every convex combination of invariant means is again an invariant mean. Thus $K_c$ is $(M,\,N)$-invariant for all $c \in [-1,1]$. This shows that not every $(M,\,N)$-invariant mean is obtained in sequences $(A_\alpha)$, $(B_\alpha)$.

\subsection{Application to functional equations} There appear a natural problem: which results known for continuous means can be adapted to the discontinuous setting? 

In this section we are going to prove just a single result inspired by Matkowski \cite[Theorem 4]{Mat06b}. 
\begin{prop}
 Let $M,\,N \colon I^2 \to I$ be two means with $M \le N$, having property \eq{weakIn}, and $\Phi \colon I^2 \to \R$ be a continuous function. Then
\Eq{Mat06bAs}{
 \Phi(x,y)=\Phi(M(x,y),N(x,y)) \qquad \text{for all }x,\,y \in I
 }
if and only if there exists a continuous function $f \colon I \to \R$ such that  
 \Eq{*}{
\Phi=f \circ \Tr.
}
Moreover if $x \mapsto \Phi(x,x)$ is an injective function then $\Tr$ is continuous.
\end{prop}
Recall that, like in many other results, comparability may be replaced by symmetry.
\begin{proof}
Take $x,\,y \in I$ arbitrarily. Using \eq{E:defxalyal}, equality \eq{Mat06bAs} can be rewritten as
$\Phi(x_\alpha,y_\alpha)=\Phi(x_{\alpha+1},y_{\alpha+1})$ for every $\alpha$.

By continuity of $\Phi$ we may extend the inductive proof to limit ordinals and 
obtain $\Phi(x_\alpha,y_\alpha)=\Phi(x_0,y_0)$ for every $\alpha$. If we put $\alpha=\omega_1$, by \eq{def:Trdev}, we obtain
\Eq{E:fTrpr}{
\Phi(\Tr(x,y),\Tr(x,y))=\Phi(x,y).
}
To complete the first implication we can simply define $f(x):=\Phi(x,x)$.

The converse implication is immediate in view of $(M,N)$-invariance of $\Tr$.

Additionally, if $x \mapsto \Phi(x,x)$ is injective, then so is $f$. In particular $f^{-1}$ exists and it is a continuous function. 

Consequently $\Tr=f^{-1} \circ \Phi$ is continuous, too.
\end{proof}

\subsection{Conclusions}
In this paper we discussed some invariant means which naturally emerged in a case of two noncontinuous means which are either comparable or both symmetric
(sometimes additionally satisfying condition \eq{weakIn}\:). 

There appear some natural problems concerning this new aspect. For example: 
{\it (i)} find out the 'noncontinuous counterpart' of results which are stated for continuous means, 
{\it (ii)} find out some additional assumption(s) to invariant mean which can be made in order to obtain the uniqueness of the solution (we presented three of those: minimality, maximality, and continuity),
{\it (iii)} generalize this concept to multivariable means (it is relatively natural in case of $\Lo$ and $\Up$ only).

Some progress toward (i) and (ii) was presented while the third aspect is  outside the scope of the present paper.
\def\cprime{$'$} \def\R{\mathbb R} \def\Z{\mathbb Z} \def\Q{\mathbb Q}
  \def\C{\mathbb C}

\end{document}